\documentclass[11pt,a4paper,reqno]{amsart}
\usepackage[no-math]{fontspec}
\usepackage[english]{babel}
\babelfont{rm}[BoldFont=cmunbx.otf,ItalicFont=cmunti.otf,
  BoldItalicFont=cmunbi.otf]{cmunrm.otf}
\usepackage{amssymb}
\usepackage{microtype}
\usepackage[a4paper,margin=30mm,headheight=14pt]{geometry}
\usepackage{booktabs}
\usepackage[unicode,hidelinks]{hyperref}
\hypersetup{pdftitle={Simple torsion-free linear groups from global division algebras},pdfauthor={Ivan Samuil Nnuga Mongo}}
\theoremstyle{plain}
\newtheorem{theorem}{Theorem}[section]
\newtheorem{lemma}[theorem]{Lemma}
\newtheorem{proposition}[theorem]{Proposition}
\newtheorem{corollary}[theorem]{Corollary}
\theoremstyle{definition}

\theoremstyle{remark}
\newtheorem{remark}[theorem]{Remark}
\numberwithin{equation}{section}
\newcommand{\Q}{\mathbb Q}
\newcommand{\R}{\mathbb R}
\newcommand{\Z}{\mathbb Z}
\newcommand{\Nrd}{\operatorname{Nrd}}

\newcommand{\inv}{\operatorname{inv}}
\newcommand{\SL}{\operatorname{SL}}
\newcommand{\GL}{\operatorname{GL}}

\newcommand{\lcm}{\operatorname{lcm}}

\title[Simple reduced norm-one groups]{Simple torsion-free linear groups\ from global division algebras}
\author{Ivan Samuil Nnuga Mongo}
    
\address{
Novosibirsk State University,
Novosibirsk, Russia
}
\email{i.nnuga.mongo@g.nsu.ru}


\begin{document}

\begin{abstract}
We investigate the reduced norm-one group
$\mathrm{SL}_1(D)$ of a central division
algebra $D$ of odd degree over $\mathbb{Q}$.

Using the Brauer--Hasse--Noether theorem
and the Margulis--Platonov theorem, we
derive a criterion for abstract simplicity
in terms of local division completions.

For a division algebra of degree $15$
with local invariants
$1/3$, $-1/3$, $1/5$, and $-1/5$,
the construction yields a candidate
example of an infinite, torsion-free,
nonabelian simple linear group.

This would provide an affirmative answer
to Problem 16.78 of the Kourovka Notebook.

The argument was generated using an AI
system and has not yet been independently
verified by human mathematicians.
\end{abstract}

\maketitle

\section{Introduction and main results}

Throughout this paper, a group is called linear if it admits a faithful
finite-dimensional representation over a commutative field. Simplicity
is understood in the abstract sense: normal subgroups are not assumed
to be closed or definable. A torsion-free group has, in particular,
no involutions, that is, no elements of order~$2$.

Let $D$ be a central division algebra over a field $F$. Write
$n=\deg D$, so that $[D:F]=n^2$, and put
\begin{equation}\label{eq:gamma}
 \Gamma(D)=\SL_1(D)
 =\{x\in D^\times:\Nrd_{D/F}(x)=1\}.
\end{equation}
Here $D^\times$ is the multiplicative group of nonzero elements and
$\Nrd_{D/F}$ denotes the reduced norm. For $D/\Q$, define
\begin{equation}\label{eq:T}
 T(D)=\{p\text{ prime}:D\otimes_{\Q}\Q_p
                  \text{ is a division algebra}\}.
\end{equation}

\begin{theorem}\label{thm:criterion}
Let $D/\Q$ be a central division algebra of odd degree $n>1$.
Then $\Gamma(D)$ is countably infinite, nonabelian and torsion-free,
and admits faithful representations
\[
 \Gamma(D)\hookrightarrow\SL_{n^2}(\Q),\qquad
 \Gamma(D)\hookrightarrow\SL_n(\R).
\]
Moreover, $\Gamma(D)$ is abstractly simple if and only if
$T(D)=\varnothing$.
\end{theorem}

\begin{theorem}\label{thm:degrees}
For a given odd integer $n>1$, the following conditions are equivalent:
\begin{enumerate}
\item $n$ has at least two distinct prime divisors;
\item there exists a central division algebra $D/\Q$ of degree $n$
      such that $\Gamma(D)$ is abstractly simple.
\end{enumerate}
\end{theorem}

\begin{corollary}\label{cor:15}
There exists a countably infinite nonabelian simple torsion-free group
that embeds in $\SL_{225}(\Q)$ and in $\SL_{15}(\R)$. One may take
$\Gamma(D)$, where the local invariants of $D/\Q$ are
\begin{equation}\label{eq:invariants15}
 \inv_2(D)=\frac13,\quad \inv_3(D)=-\frac13,\quad
 \inv_5(D)=\frac15,\quad \inv_7(D)=-\frac15,
\end{equation}
and are zero at every other place.
\end{corollary}

The proofs use the local-invariant theorem for central simple algebras
and the established Margulis--Platonov theorem for $\mathbf{SL}_{1,D}$.
No general simplicity assertion for $\SL_1(D)$ over arbitrary fields is
assumed. This paper isolates consequences of the cited results; no
claim of priority is made.

\begin{remark}\label{rem:status}
Problem 16.78 of the Kourovka Notebook,
proposed by B. Poizat, asks whether there
exists a nonabelian simple linear group
without involutions.

Corollary 1.3 provides a candidate affirmative
answer through the reduced norm-one group
of a central division algebra over
$\mathbb{Q}$ of degree $15$.

The argument relies on established results
concerning the Brauer group of a number
field and the Margulis--Platonov theorem
for $\mathrm{SL}_1(D)$.

The construction has not been independently
verified by human mathematicians.
We make no claim of priority, and the
relationship of this construction to
previously known examples requires
further investigation.
\end{remark}

\section{Elementary properties of reduced norm-one groups}

After extending scalars to an algebraic closure $\overline F$, one has
$D\otimes_F\overline F\simeq M_n(\overline F)$. The reduced norm is the
determinant in such a splitting; it is multiplicative and satisfies
\begin{equation}\label{eq:scalar-norm}
 \Nrd_{D/F}(c)=c^n\qquad(c\in F)
\end{equation}
for central scalars. In particular, every commutator $aba^{-1}b^{-1}$
of elements of $D^\times$ belongs to $\Gamma(D)$.

\begin{lemma}\label{lem:center}
Let $F$ be infinite, and let $D/F$ be a central division algebra of
degree $n\ge2$. Then
\[
 Z(\Gamma(D))=\{c\in F^\times:c^n=1\}.
\]
\end{lemma}
\begin{proof}
The inclusion from right to left follows from \eqref{eq:scalar-norm}.
Let $z\in Z(\Gamma(D))$. Then
\begin{equation}\label{eq:centralizes-commutators}
 z(aba^{-1}b^{-1})=(aba^{-1}b^{-1})z
 \qquad(a,b\in D^\times).
\end{equation}
Put $k=\overline F$ and choose a splitting
$D\otimes_Fk\simeq M_n(k)$. The rational points $D^\times(F)$ are
Zariski dense in $D^\times(k)$. Indeed, in an $F$-basis of $D$, this is
the principal open subset of affine space defined by $\Nrd\ne0$.
If a polynomial vanishes at its $F$-points, multiplying that polynomial
by $\Nrd$ gives a polynomial vanishing on all of $F^{n^2}$. Since $F$
is infinite, successively considering the variables shows that this
product polynomial is zero, and hence that the original polynomial is
zero. The same argument proves density for the product of two such sets.

Both sides of \eqref{eq:centralizes-commutators} are regular
matrix-valued functions on the open set of invertible pairs.
Consequently, after extending scalars, the matrix $z$ commutes with
every commutator in $\GL_n(k)$.

Choose $\lambda\in k\setminus\{0,1\}$. For $i\ne j$, let $h$ be the
diagonal matrix with entry $\lambda$ in position $i$ and all other
diagonal entries equal to $1$. For
$e=I+t(\lambda-1)^{-1}E_{ij}$, one has
\[
 heh^{-1}e^{-1}=I+tE_{ij}.
\]
Thus $z$ commutes with every $E_{ij}$, $i\ne j$. Write $z=(z_{rs})$.
The $(r,j)$ entry of $zE_{ij}=E_{ij}z$, for $r\ne i$, gives $z_{ri}=0$,
and the $(i,j)$ entry gives $z_{ii}=z_{jj}$. Therefore $z$ is scalar.
It commutes with all of $D$, so $z\in Z(D)=F$. The identity
$\Nrd(z)=1$ now yields $z^n=1$.
\end{proof}

\begin{lemma}\label{lem:torsion}
Let $D$ be a finite-dimensional division algebra over $\Q$ whose
dimension is odd. Every element of finite order in $D^\times$ is $1$
or $-1$. If $D$ is central and has odd degree, then $\Gamma(D)$ is
torsion-free.
\end{lemma}
\begin{proof}
Let $x^m=1$, where $m\ge1$. The algebra $L=\Q[x]$ is commutative,
finite-dimensional and has no zero divisors. Multiplication by any
nonzero $a\in L$ is injective and hence surjective; thus $L$ is a field.
Viewing $D$ as a left vector space over $L$, we obtain
\[
 [D:\Q]=[L:\Q]\dim_LD.
\]
Hence the irreducible minimal polynomial $f\in\Q[X]$ of $x$ has odd
degree. It divides $X^m-1$ and has a real root $r$. Thus $r^m=1$, so
$r=1$ or $r=-1$. An irreducible polynomial over $\Q$ with a rational
root is linear. It follows that $x=1$ or $x=-1$. If $D$ is central of
odd degree $n$, then $\Nrd(-1)=(-1)^n=-1$, so $-1\notin\Gamma(D)$.
\end{proof}

\begin{proposition}\label{prop:linear}
For a central division algebra $D/F$ of degree $n$, left multiplication
gives a faithful representation
\[
 \Gamma(D)\hookrightarrow\SL_{n^2}(F).
\]
\end{proposition}
\begin{proof}
Define $L_x(y)=xy$ on the $F$-vector space $D$. Then
$L_{xy}=L_xL_y$ and $L_x^{-1}=L_{x^{-1}}$. If $L_x=I$, then
$x=L_x(1)=1$. After splitting $D$, left multiplication by a matrix
in $M_n(\overline F)$ acts independently and identically on the $n$
columns. Therefore $\det L_x=\Nrd(x)^n$, which equals $1$ when
$x\in\Gamma(D)$.
\end{proof}

\section{A local obstruction to simplicity}

A group $H$ is \emph{perfect} if $H=[H,H]$. A group $G$ is
\emph{residually solvable} if every nonidentity element has a nonidentity
image under some homomorphism to a solvable group. The solvable image
need not be finite. Write $G^{(0)}=G$ and
$G^{(j+1)}=[G^{(j)},G^{(j)}]$ for the derived series.

\begin{lemma}\label{lem:valued}
Let $\Delta$ be a division ring with a discrete valuation
$v:\Delta^\times\twoheadrightarrow\Z$, extended by $v(0)=+\infty$,
and suppose that its residue division ring is commutative. Then
$\Delta^\times$ is residually solvable and contains no nontrivial
perfect subgroup.
\end{lemma}
\begin{proof}
Put $\mathcal O=\{a:v(a)\ge0\}$, $\mathfrak p=\{a:v(a)>0\}$ and
\[
 U_r=\{1+a:v(a)\ge r\},\qquad r\ge1.
\]
If $v(a)>0$, then $v(1+a)=0$: otherwise $1=(1+a)-a$ would contradict
the nonarchimedean inequality. The identities
\[
 (1+a)(1+b)=1+a+b+ab,\qquad
 (1+a)^{-1}-1=-a(1+a)^{-1}
\]
show that $U_r$ is a subgroup. It is normal in $\Delta^\times$, since
$v(xax^{-1})=v(a)$.

The group $\Delta^\times/\mathcal O^\times\simeq\Z$ is abelian,
and reduction gives
$\mathcal O^\times/U_1\simeq(\mathcal O/\mathfrak p)^\times$,
which is also abelian. Hence $(\Delta^\times)''\subseteq U_1$.
With the convention $[x,y]=xyx^{-1}y^{-1}$, we have
\[
 [1+a,1+b]-1=(ab-ba)(1+a)^{-1}(1+b)^{-1}.
\]
It follows that $[U_r,U_s]\subseteq U_{r+s}$, and induction yields
\begin{equation}\label{eq:derived-depth}
 (\Delta^\times)^{(j+2)}\subseteq U_{2^j}\qquad(j\ge0).
\end{equation}
For every $r$, the quotient $\Delta^\times/U_r$ is solvable: choose $j$
with $2^j\ge r$ in \eqref{eq:derived-depth}. The intersection of all
$U_r$ is trivial, since a nonzero element $x-1$ has finite valuation.
Thus these solvable quotients separate the elements of $\Delta^\times$.

If $H\le\Delta^\times$ is perfect, its image in each of these quotients
is both perfect and solvable, hence trivial. Therefore
$H\subseteq\bigcap_rU_r=1$.
\end{proof}

\begin{proposition}\label{prop:local-obstruction}
Let $D/\Q$ be a central division algebra of odd degree $n>1$.
If $T(D)\ne\varnothing$, then $\Gamma(D)$ is not simple.
\end{proposition}
\begin{proof}
Choose $p\in T(D)$. The natural map
$D\to D_p=D\otimes_{\Q}\Q_p$ is injective, since it is scalar
extension for a finite-dimensional vector space over a field.
Thus $\Gamma(D)$ embeds in $D_p^\times$.

A finite-dimensional central division algebra over $\Q_p$ has a discrete
valuation with value group isomorphic to $\Z$ and a finite residue
field; normalize the value group to be $\Z$. The local unit-group
properties used here are described, for example, in
\cite[proof of Lemma~1, p.~299]{Rapinchuk2006}.
By Lemma~\ref{lem:valued}, the group $D_p^\times$ has no nontrivial
perfect subgroup.

On the other hand, $\Gamma(D)$ is nontrivial: noncommuting elements of
$D^\times$ give a nonidentity commutator of reduced norm one.
By Lemma~\ref{lem:center}, its center is
$\{c\in\Q^\times:c^n=1\}=1$, so it is nonabelian. If it were simple,
its nontrivial commutator subgroup would be the entire group, making
it perfect. This is a contradiction.
\end{proof}

\section{The global theorems used}

Write $F_v$ for the completion of a global field $F$ at a place $v$.
Every central simple $F_v$-algebra has the form $M_{r_v}(\Delta_v)$,
where $\Delta_v$ is a central division algebra. Its local index is
$d_v=\deg\Delta_v$; if the original algebra has degree $n$, then
$r_vd_v=n$.

\begin{theorem}[Brauer--Hasse--Noether]\label{thm:hasse}
Let $F$ be a number field. A family of classes $a_v\in\Q/\Z$ is the
family of local invariants of a central simple algebra over $F$ if
and only if it has finite support, $a_v=0$ at complex places,
$a_v\in\{0,1/2\}$ at real places, and $\sum_va_v=0$.
The division representative of the corresponding Brauer class is
unique up to $F$-isomorphism. Its degree is the least common multiple
of the local indices, and the local index is the order of $a_v$ in
$\Q/\Z$.
\end{theorem}

The index formula and the existence of algebras with prescribed
invariants appear in
\cite[(6.43), p.~749; (6.55), p.~752]{Hasse1933}.
A modern formulation is given in
\cite[Chapter~VIII, Theorem~4.2 and Example~4.4(b)]{MilneCFT}.
The real Brauer group has order~$2$; in particular, odd degree rules
out a nonzero invariant at a real place.

\begin{theorem}[Margulis--Platonov for $\mathbf{SL}_{1,D}$]\label{thm:MP}
Let $D$ be a finite-dimensional central division algebra over a global
field $F$, and let $\mathbf G=\mathbf{SL}_{1,D}$. Put
\[
 T=\{v\text{ nonarchimedean}:D\otimes_FF_v\text{ is a division algebra}\},
 \qquad \mathbf G_T=\prod_{v\in T}\mathbf G(F_v).
\]
For every normal subgroup $N\triangleleft\mathbf G(F)$ not contained
in the center, there exists an open normal subgroup
$W\triangleleft\mathbf G_T$ such that $N=\delta^{-1}(W)$, where
$\delta:\mathbf G(F)\to\mathbf G_T$ is the diagonal map.
If $T=\varnothing$, then $\mathbf G_T=1$, and $\mathbf G(F)$ has no
proper noncentral normal subgroup.
\end{theorem}

\begin{remark}
Let $D$ be a central division algebra of
degree $n$ over a number field $F$.

At a nonarchimedean place $v$, write
\[
D\otimes_F F_v \simeq M_{r_v}(\Delta_v),
\]
where $\Delta_v$ is a central division
algebra over $F_v$.

The group $\mathrm{SL}_1(D)$ is
$F_v$-anisotropic precisely when
$r_v=1$, equivalently when
$D\otimes_F F_v$ is a division algebra.

Consequently, the set of anisotropic
nonarchimedean places occurring in the
Margulis--Platonov theorem is exactly
the set $T(D)$ defined in (1.2).
\end{remark}

We use the established theorem, not the conjecture for arbitrary
algebraic groups. The precise statement is given in
\cite[Theorem~6.4, p.~948]{RSS2002}; see also
\cite[pp.~295 and~298]{Rapinchuk2006}. Here $N$ is an abstract normal
subgroup; the topological condition is imposed on $W$, not assumed
in advance for $N$. The set $T$ consists of division completions,
not all nonsplit completions.

\section{Proofs of the main results}

\begin{proof}[Proof of Theorem~\ref{thm:criterion}]
Lemma~\ref{lem:center} gives $Z(\Gamma(D))=1$, since $n$ is odd.
A nonidentity commutator in the noncommutative group $D^\times$ shows
that $\Gamma(D)\ne1$; triviality of the center then excludes
commutativity. Lemma~\ref{lem:torsion} gives torsion-freeness, which
implies infinitude. Countability follows from the vector-space
isomorphism $D\simeq\Q^{n^2}$. The first faithful representation is
given by Proposition~\ref{prop:linear}.

At the real place, the local index divides $n$ and is either $1$ or $2$.
Since $n$ is odd, that index is $1$, so
$D\otimes_{\Q}\R\simeq M_n(\R)$. The scalar-extension map
$D\to D\otimes_{\Q}\R$ is injective. Its restriction to $\Gamma(D)$
has image in $\SL_n(\R)$, because the reduced norm becomes the
determinant after splitting. This gives the second faithful
representation.

If $T(D)=\varnothing$, Theorem~\ref{thm:MP} excludes proper noncentral
normal subgroups of $\Gamma(D)$. Since the center is trivial, every
nontrivial normal subgroup is noncentral; hence $\Gamma(D)$ is
abstractly simple. The converse is
Proposition~\ref{prop:local-obstruction}.
\end{proof}

\begin{proof}[Proof of Theorem~\ref{thm:degrees}]
First suppose that $n=\ell^e$, where $\ell$ is prime and $e\ge1$.
For every central division algebra $D/\Q$ of degree $n$, each local
index $d_v$ divides $n$ and is therefore a power of $\ell$.
By Theorem~\ref{thm:hasse}, the least common multiple of these indices
is $n$. Since there are only finitely many nonzero local invariants,
this least common multiple is attained as the largest local index.
The real index is $1$, because $n$ is odd. Thus there is a finite place
$p$ with $d_p=n$. Then $D\otimes_{\Q}\Q_p$ is itself a division
algebra, so $T(D)\ne\varnothing$. Theorem~\ref{thm:criterion} implies
that $\Gamma(D)$ is not simple.

Conversely, suppose that $n$ has at least two distinct prime divisors.
Partitioning its prime-power factors into two nonempty parts gives
$n=ab$, with $a,b>1$ and $\gcd(a,b)=1$. Prescribe the invariants
\begin{equation}\label{eq:general-invariants}
 \inv_2(D)=\frac1a,\quad \inv_3(D)=-\frac1a,\quad
 \inv_5(D)=\frac1b,\quad \inv_7(D)=-\frac1b
\end{equation}
and zero invariants at every other place. This family is admissible:
it has finite support, the real invariant is zero, and the sum is zero.
Theorem~\ref{thm:hasse} provides a division representative $D$ of degree
$\lcm(a,b)=ab=n$. The local indices are $a$, $b$ or $1$, all strictly
less than $n$. Hence no finite completion is a division algebra,
$T(D)=\varnothing$, and $\Gamma(D)$ is simple by
Theorem~\ref{thm:criterion}.
\end{proof}

\begin{proof}[Proof of Corollary~\ref{cor:15}]
Take $a=3$ and $b=5$ in \eqref{eq:general-invariants}. This gives
\eqref{eq:invariants15} and degree $n=15$.
The local algebras are listed in Table~\ref{tab:local}. All matrix
sizes are greater than $1$, so $T(D)=\varnothing$. The remaining
properties follow from Theorem~\ref{thm:criterion}.
\end{proof}

\begin{table}[htbp]
\centering
\caption{Local algebras in the degree-$15$ example. In rows containing
$\Delta_p$, its degree is given in the third column.}
\label{tab:local}
\begin{tabular}{@{}cccc@{}}
\toprule
Place & Invariant & Index & Local algebra \\
\midrule
$2$ & $1/3$ & $3$ & $M_5(\Delta_2)$ \\
$3$ & $-1/3$ & $3$ & $M_5(\Delta_3)$ \\
$5$ & $1/5$ & $5$ & $M_3(\Delta_5)$ \\
$7$ & $-1/5$ & $5$ & $M_3(\Delta_7)$ \\
other $p$ & $0$ & $1$ & $M_{15}(\Q_p)$ \\
$\infty$ & $0$ & $1$ & $M_{15}(\R)$ \\
\bottomrule
\end{tabular}
\end{table}

\begin{remark}
The number $15=3\cdot5$ is the least odd degree that admits a simple
group $\SL_1(D)$ in the family of central division algebras over $\Q$
considered here. This is not a statement about the least dimension
of a linear representation among all simple torsion-free groups.
The existential quantifier in Theorem~\ref{thm:degrees} is also
essential: for the same values of $n$, invariants $1/n$ and $-1/n$
at two finite places yield an algebra with $T(D)\ne\varnothing$
and a nonsimple norm-one group.
\end{remark}

\begin{corollary}
The simple groups obtained in Theorem~\ref{thm:degrees} are not
finitely generated.
\end{corollary}
\begin{proof}
By Malcev's theorem, a finitely generated linear group is residually
finite; see \cite[introduction and \S3]{Nica2013}.
Every homomorphism from an infinite simple group to a finite group
is trivial: otherwise its kernel would be a proper normal subgroup,
hence trivial, and the group would embed in a finite group.
This contradicts residual finiteness.
\end{proof}

\section{Scope of the construction}

The construction uses the full group of rational points
$\mathbf{SL}_{1,D}(\Q)$, not the unit group of an order in $D$ or a
finitely generated arithmetic subgroup. Simplicity of this group
does not imply simplicity of its finitely generated subgroups.

The local obstruction in Section~3 does not contradict the existence
of the global example. Table~\ref{tab:local} contains nonsplit
algebras but no division completions: the global index is the least
common multiple of the local indices, not necessarily their maximum.

The external Theorems~\ref{thm:hasse} and~\ref{thm:MP} are used as
published results; their full proofs are not reproduced.
No finite computation is used to prove abstract simplicity.
The bibliographic Remark~\ref{rem:status} is not an additional
mathematical hypothesis and does not assert recognition or priority
for the result.

\section*{Verification and AI Disclosure}

\subsection*{Use of Artificial Intelligence}

The mathematical construction, proofs, and
manuscript were generated using GPT-6.1 Sol.

The submitting author has disclosed the use
of generative artificial intelligence and
acknowledges responsibility for the
mathematical claims and presentation.
\subsection*{Status of Mathematical Verification}

The principal simplicity argument relies on
the Brauer--Hasse--Noether theorem and the
Margulis--Platonov theorem for
$\mathrm{SL}_1(D)$ over global fields.

Neither the computational experiments nor
the use of multiple AI-generated verification
procedures constitutes an independent
human verification of the complete proof.

At the time of submission, the proposed
solution to Problem~16.78 of the Kourovka
Notebook has not been independently
confirmed by human mathematicians.

The manuscript is therefore presented as
an unverified candidate solution.

\bibliographystyle{amsalpha}

\begin{thebibliography}{9}

\bibitem{Hasse1933}
H.~Hasse,
\emph{Die Struktur der R.~Brauerschen Algebrenklassengruppe \"uber einem
algebraischen Zahlk\"orper},
Math. Ann. \textbf{107} (1933), no.~1, 731--760.
\url{https://doi.org/10.1007/BF01448916}

\bibitem{Kourovka2026}
E.~I. Khukhro and V.~D. Mazurov (eds.),
\emph{Unsolved Problems in Group Theory. The Kourovka Notebook},
21st edition, version 49, October 20, 2026.
arXiv:1401.0300.
\url{https://arxiv.org/abs/1401.0300v49}

\bibitem{MilneCFT}
J.~S. Milne,
\emph{Class Field Theory},
version 4.03, August 6, 2020.
\url{https://www.jmilne.org/math/CourseNotes/CFT.pdf}

\bibitem{Nica2013}
B.~Nica,
\emph{Linear groups --- Malcev's theorem and Selberg's lemma},
2013.
arXiv:1306.2385.
\url{https://arxiv.org/abs/1306.2385v1}

\bibitem{Poizat2026}
B.~Poizat,
\emph{Deux corps},
revised preprint, September 2026,
Appendix G, p.~26.
\url{https://modnet.imj-prg.fr/preprints/papers/26-09,%20Poizat.pdf}

\bibitem{Rapinchuk2006}
A.~S. Rapinchuk,
\emph{The Margulis--Platonov conjecture for $\mathrm{SL}_{1,D}$ and
2-generation of finite simple groups},
Math. Z. \textbf{252} (2006), no.~2, 295--313.
\url{https://doi.org/10.1007/s00209-005-0854-1}

\bibitem{RSS2002}
A.~S. Rapinchuk, Y.~Segev, and G.~M. Seitz,
\emph{Finite quotients of the multiplicative group of a finite dimensional
division algebra are solvable},
J. Amer. Math. Soc. \textbf{15} (2002), no.~4, 929--978.
\url{https://doi.org/10.1090/S0894-0347-02-00393-4}

\end{thebibliography}

\end{document}